\documentclass{article}
\usepackage{fullpage}

\usepackage{amsmath,amsthm,graphicx,amsfonts,lineno,bm}

\graphicspath{{figs/}}

\newcommand{\rec}{\overline{\operatorname{cr}}}

\newcommand{\RR}{\ensuremath{\mathbb R}}  
\newcommand{\ZZ}{\ensuremath{\mathbb Z}}  
\newcommand{\EE}{\ensuremath{\mathbb E}}  

\DeclareMathOperator{\area}{area}

\newtheorem{theorem}{Theorem}

\newtheorem{lemma}[theorem]{Lemma}

\title{A Note on the 2-Colored Rectilinear Crossing Number of Random Point Sets in the Unit Square}

\begin{document}
\author{%
Sergio Cabello\thanks{Faculty of Mathematics and Physics, University of Ljubljana, Slovenia, 
and Institute of Mathematics, Physics and Mechanics, Slovenia.
\texttt{sergio.cabello@fmf.uni-lj.si}.}
\and 
\'Eva Czabarka\thanks{Department of Mathematics, University of South Carolina, Columbia, \texttt{czabarka@math.sc.edu}} 
\and
Ruy Fabila-Monroy \thanks{Departamento de Matem\'aticas, Cinvestav. \texttt{ruyfabila@math.cinvestav.edu.mx}} 
\and
Yuya Higashikawa \thanks{Graduate School of Information Science, University of Hyogo, Japan. \texttt{higashikawa@gsis.u-hyogo.ac.jp}} 
\and
Raimund Seidel\thanks{Saarland University, Germany. \texttt{rseidel@cs.uni-saarland.de}}
\and
L\'aszl\'o Sz\'ekely \thanks{Department of Mathematics, University of South Carolina, Columbia, \texttt{szekely@math.sc.edu}} \and
Josef Tkadlec \thanks{Computer Science Institute, Charles University, Czech Republic, \texttt{josef.tkadlec@iuuk.mff.cuni.cz}}
\and
Alexandra Wesolek
\thanks{Department of Mathematics, Technische Universit\"at Berlin, Germany, \texttt{wesolek@tu-berlin.de}
}
}
\maketitle

\begin{abstract}
Let $S$ be a set of four points chosen independently, uniformly at random from a square. Join every pair of points of $S$ with a straight line segment. Color these edges red if they have positive slope and blue, otherwise. We show that the probability that $S$ defines a pair of crossing edges of the same color is equal to $1/4$. This is connected to a recent result of Aichholzer et al. [GD 2019] who showed that by 2-colouring the edges of a geometric graph and counting monochromatic crossings instead of crossings, the number of crossings can be more than halfed. Our result shows that for the described random drawings, there is a coloring of the edges such that the number of monochromatic crossings is in expectation $\frac{1}{2}-\frac{7}{50}$ of the total number of crossings.
\end{abstract}

\section{Introduction}

Let $R$ be a closed bounded convex set in the plane, and let $q(R)$ be the probability that four points chosen uniformly at random from $R$ form the vertices of a convex quadrilateral. The problem of computing $q(R)$ was proposed by Sylvester~\cite{4point} in 1868 and became known as \emph{Sylvester's Four-Point Problem}.
Woolhouse~\cite{square} showed that if $R$ is a square then
\[q(R)=\frac{25}{36}.\]

A \emph{geometric graph} is a graph whose vertex set is a set of points in general position in the plane,
and its edges are straight line segments joining these points. Let $D$ be a geometric graph on $n$ vertices in general position. The \emph{number of crossings} of $D$ is the number of pairs of its edges that intersect in their interior. We denote this number with $\rec (D)$. 
Linearity of expectation implies that, if the vertices of $D$ are chosen uniformly and independently at random from $R$, then \[ \EE[\rec (D)] = q(R)\cdot \binom {n}{4}.\]

Let $\chi$ be an edge coloring of $D$. Let $\rec (D,\chi)$ be the number of pairs of edges  of $D$
of the same color that cross. If $\chi$ is a random edge coloring of $D$, in which every edge
is assigned one of two colors with probability $1/2$, then linearity of expectation implies that
\[\EE[\rec (D,\chi)] = \frac{1}{2} \cdot \rec(D).\] 
Recently, Aichholzer et al.~\cite{2colored}
showed that there exists a constant $c >0$, such that if $D$ is a complete geometric graph, then there exists
an edge $2$-coloring, $\chi$, of $D$ such that 
\[\rec (D,\chi) = \left (\frac{1}{2}-c \right ) \cdot \rec(D).\] 
In this paper, we consider the case when the vertices of a complete geometric graph $D$ are chosen at random from the unit square $[0,1] \times [0,1].$ We show the following theorem.
\begin{theorem} \label{thm:main}
Let $S$ be a set of four points chosen independently and uniformly at random from the unit square.
Join every pair of points of $S$ with a straight line segment. Color each such edge red if it has positive slope
and blue otherwise\footnote{With this convention, horizontal edges (slope 0) and vertical edges (undefined slope) are blue.
In any case, these boundary cases are negligible.}. 
Then the probability that $S$ defines a pair of
crossing edges of the same color is equal to $1/4$. 
\end{theorem}

Let $\chi_{slope}$ be the edge coloring of $D$ in which an edge is colored red if it has positive
slope and blue, otherwise. Theorem~\ref{thm:main} implies that if the vertices of $D$ are chosen
uniformly at random from a square then
\[\EE[\rec (D,\chi_{slope})] ~=~ \frac{1}{4} \cdot \binom{n}{4} ~=~ \frac{18}{50}\cdot \EE[\rec (D)] ~=~ \left ( \frac{1}{2}-\frac{7}{50} \right)\EE[\rec (D)] .\] 
The constant $7/50$ is significantly larger than the constant $c$ of Aichholzer et al.~\cite{2colored} for the generic case.
Incidentally, this idea of assigning different colors to edges depending on the slope was used by Bla{\v z}ek and Koman~\cite{2page}, to obtain a (conjectured) crossing optimal drawing of the complete graph $K_n$ as follows. Place
$n$ points in a regular polygon; for every pair of vertices $u$ and $v$, if the line segment
$\overline{uv}$ has positive slope then draw this edge as diagonal of the polygon, otherwise draw this edge as a chord on the outside of the polygon.


\section{Preliminaries}
For any non-negative integer $n$, let $[n]=\{0,\dots, n\}$.

For non-negative integers $w$ and $h$, 
we use $Q_{w,h}$ to denote the integer grid $[w]\times [h]$, which contains $(w+1)(h+1)$ points.
The four \emph{sides} of $Q_{w,h}$ are the subsets $\{0\}\times [h]$ (\emph{left side}), 
$\{w\}\times [h]$ (\emph{right side}), $[w]\times \{ 0 \}$ (\emph{bottom side}), and $[w]\times \{ h \}$ (\emph{top side}).
The four \emph{corners} of $Q_{w,h}$ are the points that lie on two sides simultaneously: $\{ 0,w\}\times \{0,h\}$.
See Figure~\ref{fig:Q}.
For a set of points $S\subset Q_{w,h}$, we say that its bounding box is $Q_{w,h}$ if 
$S$ intersects each of the four sides of $Q_{w,h}$. (A corner intersects two sides.)

\begin{figure}\
 	\centering
 	\includegraphics[page=1,scale=1.2]{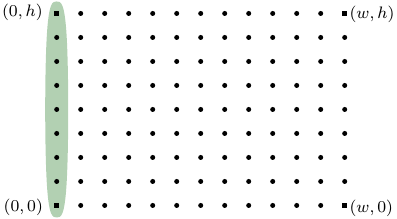}
 	\caption{The grid $Q_{w,h}$ for $w=12$ and $h=8$. Its corners are marked with boxes. Its left side
		is slightly shaded in green.}
	\label{fig:Q}
 \end{figure}
 
We define $P_m:= Q_{m,m}$ for each non-negative integer $m$.

We use $A=B\pm C$ as a shorthand for $B-C\le A\le B+C$.
In our setting, $C$ is given in $O(\cdot)$ notation.
Thus, writing $A(n)=B(n)\pm O(n)$ means that there is a constant $c>0$
such that $B(n)-cn \le A(n)\le B(n)+cn$ for all $n$.

To estimate the number of integer points inside a convex region of the plane,
there is a tight, classical bound given by Nosarzewska~\cite{Nosarzewska},  
which we simplify to the following rough statement:
\begin{theorem}
\label{thm:Nosarzewska}
	Let $K$ be a convex and compact set in the plane with area $A$ and perimeter $L$.
	Then $|\ZZ^2\cap K|= A\pm O(1+L)$.
\end{theorem} 

We will also use the following estimates for sums, whose bounds follow from the closed
form or from approximating by integrals. 
\begin{lemma}
	For each positive integer $n$ we have
	\begin{align*}
		&\sum_{i=0}^n i~=~ n^2/2 \pm O(n) , ~~~~~
		\sum_{i=0}^n i^2~=~ n^3/3 \pm O(n^2), ~~~~~
		\sum_{i=0}^n i^3~=~ n^4/4 \pm O(n^3), \\
		&\sum_{i=0}^n i(n-i)~=~ n^3/6 \pm O(n^2), ~~~~~
		\sum_{i=0}^n i^2(n-i)~=~ \sum_{i=0}^n i(n-i)^2~=~ n^4/12 \pm O(n^3).
	\end{align*}

\end{lemma}

\section{Proof of Theorem~\ref{thm:main}}

Valtr~\cite{pavel} proved that the probability that $n$ points chosen independently and uniformly at random
from a parallelogram are in convex position is equal to $\left( \binom{2n-2}{n-1}/n! \right )^2.$ 
Setting $n=4$ in this formula, one recovers the value $q(R)=25/36$ for a unit square $R$.
A very high level overview of his approach is as follows. 
First consider the integer grid $P_m$ and assume that the $n$ points are chosen from this grid. 
Note that every set of $n$ points in convex position has a unique isothetic bounding box. Count the number of convex polygons of $n$ vertices having a fixed integer grid $Q_{w,h}$ as bounding box. 
Then, sum this number over all possible copies of all possible integer grids $Q_{w,h}$ that fit inside $P_m$. 
This gives the number of convex polygons with $n$ vertices in the integer grid $P_m$. 
Comparing this to the number of $n$-subsets of the integer grid $P_m$ and letting $m$ tend to infinity, one
gets the result. 

We follow a similar paradigm. However, in contrast to the work of Valtr, we use estimates to approximate the sums
that appear in the analysis. The error introduced by approximating the sums is negligible and vanishes
when we let $m$ tend to infinity. 
Valtr uses a clever technique to count the number of convex polygons on $n$ vertices when $Q_{w,h}$ is the bounding box. We introduce a different way of counting these polygons when $n=4$.  This new counting allows us to distinguish convex polygons depending on whether their diagonals are monochromatic or bichromatic, for which Valtr's technique does not seem suitable.

For any positive integers $w$ and $h$,  
let $A_{w,h}$ be the number of sets $S \in \binom{Q_{w,h}}{4}$ satisfying the following:
\begin{itemize}
 \item the bounding box of $S$ is precisely $Q_{w,h}$; 
 \item $S$ is in convex position; and
 \item the diagonals of the convex quadrilateral defined by $S$ are of the same color.
\end{itemize}

We are interested in estimating $A_{w,h}$. 
We decompose the relevant sets depending on the number of points placed at the corners of the bounding box. 
More precisely, for each $i=0,\dots,4$, let $A_{w,h}^{(i)}$ be the number of sets 
$S \in \binom{Q_{w,h}}{4}$ satisfying the following:
\begin{itemize}
 \item the bounding box of $S$ is precisely $Q_{w,h}$; 
 \item $S$ is in convex position;
 \item the diagonals of the convex quadrilateral defined by $S$ are of the same color; and
 \item exactly $i$ of the points of $S$ are corners of $Q_{w,h}$.
\end{itemize}
Obviously,
\[
	A_{w,h} ~=~ \sum_{i=0}^4 A_{w,h}^{(i)}.
\]
We compute estimates for each of the values $A_{w,h}^{(i)}$ separately.  
In some of our estimates we neglect the case where the diagonals are horizontal or vertical, as they contribute
a negligible part. 

We will use $p_1,\dots,p_4$ to denote the points in the sets $S \in \binom{Q_{w,h}}{4}$ we consider. 
For each point $p_i\in S$, we use $(x_i,y_i)$ for its coordinates.

\begin{itemize}

\item {$\bm{A_{w,h}^{(4)}}$ and $\bm{A_{w,h}^{(3)}}$}\\
We have $A_{w,h}^{(4)}=1$ and $A_{w,h}^{(3)}=O(wh)$;
they will be negligible in the final computation.

\item {$\bm{A_{w,h}^{(0)}}$}\\
In this case, the four points of $S$ are on the sides of $Q_{w,h}$ and none of them
lie in a corner. Without loss of generality,
we assume that $p_1, p_2, p_3, p_4$ lie on the top, bottom, left and right sides of $Q_{w,h}$, 
respectively. 
See Figure~\ref{fig:a0}.
Thus, the diagonals that cross are $\overline{p_1p_2}$ and $\overline{p_3p_4}$. 
There are $(w-1)^2$ choices for $\overline{p_1p_2}$ and $(h-1)^2$ choices for $\overline{p_3p_4}$,
as none of the points can be a corner. Exactly\footnote{Here we neglect the case of 
of $\overline{p_1p_2}$ being vertical or $\overline{p_3p_4}$ being horizontal, which is covered by the term
$\pm O(w^2h+wh^2)$ in the computation below.} half of the choices for  
$\overline{p_1p_2}$ are red, exactly half of the choices of $\overline{p_3p_4}$ are red, and those choices are independent. Thus,
 \[
	A_{w,h}^{(0)} ~=~ \frac{(w-1)^2 (h-1)^2}{2} \pm O(w^2h+wh^2) ~=~ \frac{w^2 h^2}{2} \pm O(w^2h+wh^2).
\]
\begin{figure}\
 	\centering
 	\includegraphics[page=2,scale=1.1]{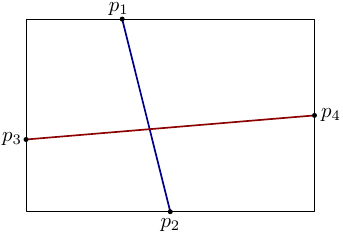}
 	\caption{The case $A_{w,h}^{(0)}$.}
	\label{fig:a0}
 \end{figure}
  
\item {$\bm{A_{w,h}^{(1)}}$}\\
Without loss of generality assume that $p_1$ is the point in a corner of $Q_{w,h}$. 
Suppose that $p_1$ is in the bottom-left corner. The other three cases are analogous. 
To have bounding box $Q_{w,h}$, one point of $S$, say $p_2$, is in the right side of $Q_{w,h}$, 
and another point of $S$, say $p_3$, is in the top side of $Q_{w,h}$. 
Note that $\overline{p_1p_2}$, $\overline{p_1p_3}$ are red, while $\overline{p_2p_3}$ is blue.
See Figure~\ref{fig:a1}.
For $S$ to define a pair of crossing edges of the same color, the remaining point of $S$, $p_4$, 
must lie in $K$ or $K'$, where:
\begin{itemize}
	\item $K=K(p_2,p_3)$ is the region above the line supporting $\overline{p_1p_3}$, below the horizontal line through $p_2$,
		and to the right of the vertical line through $p_1$;
	\item $K'=K'(p_2,p_3)$ is the region below the line supporting $\overline{p_1p_2}$, to the left of the vertical line
		through $p_3$, and above the horizontal line through $p_1$. 
\end{itemize}
Note that $K$ and $K'$ are interior disjoint triangles.
By considering the four possible corners for the point $p_1$ and the possible locations
of $p_2$ and $p_3$, we have 
\[
	A_{w,h}^{(1)} ~=~ 4 \sum_{y_2=1}^{h-1} \sum_{x_3=1}^{w-1} \Bigl( |Q_{w,h}\cap K|+|Q_{w,h}\cap K'| \pm O(1) \Bigr).\]
Since the regions $K$ and $K'$ are convex and have perimeter $O(w+h)$, we can use Theorem~\ref{thm:Nosarzewska}
to obtain
\begin{align*}
	A_{w,h}^{(1)} ~&=~ 4 \sum_{y_2=1}^{h-1} \sum_{x_3=1}^{w-1} \Bigl( \area(K) + \area(K') \pm O(w+h) \Bigr)\\
				  ~&=~ 4 \sum_{y_2=1}^{h-1} \sum_{x_3=1}^{w-1} \left( \frac{1}{2}\cdot y_2 \cdot \frac{y_2}{h}x_3 + \frac{1}{2}\cdot x_3\frac{x_3}{w}y_2 \pm O(w+h) \right)\\
				  ~&=~ 4 \sum_{y_2=1}^{h-1} \left[\Bigl( \frac{w^2}{2}\pm O(w)\Bigr) \frac{y_2^2}{2h} + \Bigl( \frac{w^3}{3}\pm O(w^2)\Bigr) \frac{y_2}{2w} \pm O(w^2+wh)\right]\\
				  ~&=~ 4 \left[\Bigl( \frac{w^2}{2}\pm O(w)\Bigr) \Bigl( \frac{h^3}{3}\pm O(h^2)\Bigr) \frac{1}{2h} + \Bigl( \frac{w^3}{3}\pm O(w^2)\Bigr) \Bigl( \frac{h^2}{2}\pm O(h)\Bigr) \frac{1}{2w} \pm O(w^2h+wh^2)\right]\\
				  ~&=~ \Bigl( w^2\pm O(w)\Bigr) \Bigl( \frac{h^2}{3}\pm O(h)\Bigr) + \Bigl( \frac{w^2}{3}\pm O(w)\Bigr) \Bigl( h^2 \pm O(h)\Bigr)  \pm O(w^2h+wh^2)\\
				  ~&=~ \frac{2}{3} w^2 h^2 \pm O(w^2h+wh^2).
\end{align*}

\begin{figure}\
 	\centering
 	\includegraphics[page=3,scale=1.1]{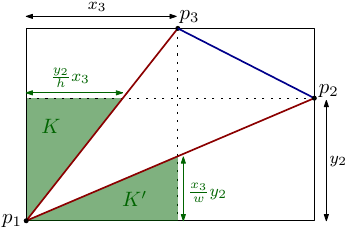}
 	\caption{The case $A_{w,h}^{(1)}$ with a point at the bottom-left corner.}
	\label{fig:a1}
\end{figure}

\item {$\bm{A_{w,h}^{(2)}}$}\\
If two of the points of $S$ lie on two corners on the same side of $Q_{w,h}$, then at least one other point of $S$
has to lie on the boundary of $Q_{w,h}$, and there are $O(w+h)$ choices for that other point.
The fourth point of $S$ can then lie anywhere in $Q_{w,h}$. In total there are $O(w^2h+wh^2)$ such
sets with two points of $S$ on the two corners of a single side of $Q_{w,h}$.

It remains to count the number of sets $S$ with two points on opposite corners of $Q_{w,h}$. 
Suppose that one point of $S$, say $p_1$, is in the bottom-left corner and another point, say $p_2$,
is in the top-right corner. 
The case when the points are in the other pair of opposing corners is analogous.
Note that the segment $\overline{p_1p_2}$ is red. See Figure~\ref{fig:a2}.

\begin{figure}\
 	\centering
 	\includegraphics[page=4,scale=1.1]{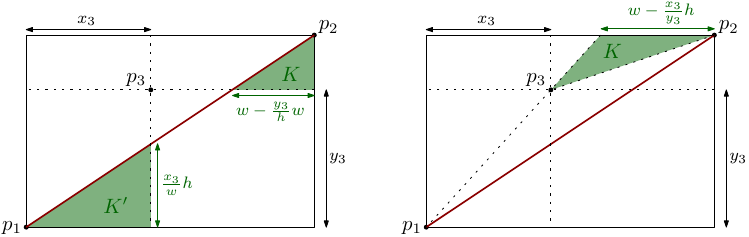}
 	\caption{The case $A_{w,h}^{(2)}$ with points at the bottom-left and top-right corner. 
		The left figure is to analyze $C_{w,h}^{(2)}$ and the right figure for $D_{w,h}^{(2)}$.}
	\label{fig:a2}
\end{figure}

Let $p_3$ and $p_4$ be the remaining points of $S$. Let $\ell$ be the line supporting the diagonal $\overline{p_1p_2}$.
We count separately the sets $S$ with $p_3$ and $p_4$ on the same side of the line $\ell$ and those with those points on opposite sides of $\ell$. More precisely, we define:
\begin{itemize}
	\item Let $C_{w,h}^{(2)}$ be the number of sets $\{p_1,p_2,p_3, p_4\}\in \binom{Q_{w,h}}{4}$ 
		contributing to $A_{w,h}^{(2)}$	such that $p_1$ is the bottom-left corner, $p_2$ is the top-right corner, 
		and $p_3,p_4$ are on opposite sides of the line supporting $\overline{p_1p_2}$. 
	\item Let $D_{w,h}^{(2)}$ be the number of sets $\{p_1,p_2,p_3, p_4\}\in \binom{Q_{w,h}}{4}$ 
		contributing to $A_{w,h}^{(2)}$ such that $p_1$ is the bottom-left corner, $p_2$ is the top-right corner, 
		and $p_3,p_4$ are on the same side of the line supporting $\overline{p_1p_2}$. 
\end{itemize}
We then have
\[
	A_{w,h}^{(2)} ~=~ 2\cdot \Bigl( C_{w,h}^{(2)}+D_{w,h}^{(2)} \Bigr),
\]
where the factor $2$ comes from choosing $p_1,p_2$ as the endpoints of the other diagonal.

Let us estimate $C_{w,h}^{(2)}$. 
Without loss of generality, let us denote by $p_3$ the point above $\ell$, and thus $p_4$ is below $\ell$.
See Figure~\ref{fig:a2}, left.
For each choice of $p_3$ above $\ell$, if the segment $\overline{p_3p_4}$ is to be red, 
the point $p_4$ must lie in $K$ or $K'$, where:
\begin{itemize}
	\item $K=K(p_3)$ is the region below $\ell$, above the horizontal line through $p_3$, 
		and to the left of the vertical line through $p_2$;
	\item $K'=K'(p_3)$ is the region below $\ell$, to the left of the vertical line through $p_3$, 
		and above the horizontal line through $p_1$. 
\end{itemize}
Note that $K$ and $K'$ are interior disjoint triangles.
Thus, 
\begin{align*}
C_{w,h}^{(2)} & = \sum_{p_3 \text{ above } \ell}  \Bigl( |Q_{w,h}\cap K|+|Q_{w,h}\cap K'| \pm O(w+h) \Bigr) \\
              & = \sum_{y_3=1}^h ~~\sum_{x_3=0}^{\lfloor y_3 w/h\rfloor} \Bigl( \area(K)+\area(K') \pm O(w+h) \Bigr)\\
              & = \sum_{y_3=1}^h ~~\sum_{x_3=0}^{\lfloor y_3 w/h\rfloor} \left( \frac{1}{2}\cdot w(1- \frac{y_3}{h})\cdot (h-y_3) +\frac{1}{2} \cdot x_3 \cdot \frac{x_3}{w}h  \pm O(w+h) \right)\\
              & = \sum_{y_3=1}^h ~~\sum_{x_3=0}^{\lfloor y_3 w/h\rfloor} \left( \frac{w}{2h} \cdot (h-y_3)^2 +\frac{h}{2w} \cdot x_3^2  \pm O(w+h) \right)\\
              & = \sum_{y_3=1}^h \left[  \frac{w}{2h} \cdot (h-y_3)^2 \cdot \Bigl( y_3 w/h \pm O(1)\Bigr) +\frac{h}{2w} \cdot \Bigl(\frac{(y_3 w/h)^3}{3} \pm O((y_3 w/h)^2)\Bigr)   \pm O(w^2+wh) \right] \\
              & = \sum_{y_3=1}^h \left[  \frac{w^2}{2h^2} \cdot (h-y_3)^2 y_3 +\frac{w^2}{6h^2} \cdot y_3^3 \pm O(w^2+wh) \right] \\
              & = \frac{w^2}{2h^2} \cdot \Bigl( \frac{h^4}{12} \pm O(h^3)\Bigr) +\frac{w^2}{6h^2} \cdot \Bigl( \frac{h^4}{4} \pm O(h^3)\Bigr) \pm O(w^2h+wh^2)  \\
			  & = \frac{1}{12} w^2h^2 \pm O(w^2h+wh^2).
\end{align*}

We continue now estimating $D_{w,h}^{(2)}$. 
Let us consider the case when both points $p_3$ and $p_4$ are above $\ell$; the other case is analogous.
Because of symmetry, we can denote by $p_3$ the point with smallest $x$-coordinate.
See Figure~\ref{fig:a2}, right.
In this case, for each choice of $p_3$, the point $p_4$ must lie in $K=K(p_3)$:
the region below the line supporting $\overline{p_1p_3}$, above the line supporting $\overline{p_2p_3}$,
and below the horizontal line through $p_2$.
This region is a triangle and, using that the slope of $\overline{p_1p_3}$ is $y_3/x_3$,
we obtain that $K$ has base (on $y=h$) equal to $w-\frac{x_3}{y_3}h$ and height equal to $h-y_3$.
See Figure~\ref{fig:a2}, right.
Therefore
\begin{align*}
D_{w,h}^{(2)} & = 2\sum_{p_3 \text{ above } \ell}  \Bigl( |Q_{w,h}\cap K| \pm O(w+h) \Bigr) \\
              & = 2 \sum_{y_3=1}^h ~~\sum_{x_3=0}^{\lfloor y_3 w/h\rfloor} \Bigl( \area(K) \pm O(w+h) \Bigr)\\
              & = 2 \sum_{y_3=1}^h ~~\sum_{x_3=0}^{\lfloor y_3 w/h\rfloor} \left( \frac{1}{2}\cdot (w- \frac{x_3}{y_3}h) \cdot (h-y_3) \pm O(w+h) \right)\\
              & = \sum_{y_3=1}^h ~~\sum_{x_3=0}^{\lfloor y_3 w/h\rfloor} \left( w(h- y_3) - x_3 (\frac{h^2}{y_3} - h) \pm O(w+h) \right)\\
              & = \sum_{y_3=1}^h \left( w(h- y_3)\Bigl( y_3 w/h \pm O(1)\Bigr) - \Bigl( \frac{(y_3 w/h)^2}{2} \pm O(y_3 w/h)\Bigr) (\frac{h^2}{y_3} - h) \pm O(w^2+wh) \right)\\
              & = \sum_{y_3=1}^h \left[ \Bigl( \frac{w^2}{h}(h-y_3)y_3 \Bigr) - \Bigl( \frac{w^2}{2} y_3 \Bigr)
					+ \Bigl( \frac{w^2}{2h} y_3^2 \Bigr) \pm O(w^2+wh)\right]\\
              & = \Bigl( \frac{w^2}{h} \bigl( \frac{h^3}{6}\pm O(h^2) \bigr) \Bigr) - \Bigl( \frac{w^2}{2} \Bigl( \frac{h^2}{2}\pm O(h) \Bigr) \Bigr)
					+ \Bigl( \frac{w^2}{2h} \Bigl( \frac{h^3}{3}\pm O(h^2) \Bigr) \Bigr) \pm O(w^2h+wh^2)\\
			  & = \frac{1}{12}w^2h^2\pm O(w^2h+wh^2).
\end{align*}

Using the computed values we conclude that
\[
	A_{w,h}^{(2)} ~=~ 2\cdot \left( \frac{1}{12}w^2h^2\pm O(w^2h+wh^2) + \frac{1}{12}w^2h^2\pm O(w^2h+wh^2) \right)
				  ~=~ \frac{1}{3}w^2h^2 \pm O(w^2h+wh^2).
\]
\end{itemize}

This finishes the estimates of each single value $A_{w,h}^{(i)}$.
Adding them we have that
\[
	A_{w,h} ~=~ w^2 h^2 \left( \frac{1}{2}+ \frac{2}{3} + \frac{1}{3}\right) \pm O(w^2h+wh^2) ~=~ \frac{3}{2}w^2 h^2 \pm O(w^2h+wh^2).
\]
We summarize our findings.
\begin{lemma}
	There are $\frac{3}{2}w^2 h^2 \pm O(w^2h+wh^2)$ sets $S\in \binom{Q_{w,h}}{4}$ such that:
	the points of $S$ are in convex position, the bounding box of $S$ is $Q_{w,h}$,
	and both diagonals of the convex quadrilateral defined by $S$ have positive slope or both diagonals have negative slope.
\end{lemma}

We have to consider now all copies of $Q_{w,h}$ that are contained in $P_m$, for all possible values $w$ and $h$.
This is the main object of interest in our approach.

\begin{lemma}
	There are $\frac{m^8}{96} \pm O(m^7)$ sets $S\in \binom{P_m}{4}$ such that:
	the points of $S$ are in convex position, 
	and both diagonals of the convex quadrilateral defined by $S$ have positive slope or both diagonals have negative slope.
\end{lemma}
\begin{proof}
	Note that there are $(m-w)\cdot (m-h)$ copies of $Q_{w,h}$ in $P_m$ because 
	the position of a corner of $Q_{w,h}$ inside $P_m$ determines the whole copy. 
	Therefore, using the loose bound $O(w^2h+wh^2)=O(m^3)$ in
	the estimate for $A_{w,h}$, we have 
	\begin{align*}
		\sum_{Q_{w,h}\text{ inside }P_m} ~~ A_{w,h} ~&=~ \sum_{w=1}^m \sum_{h=1}^m \left[ \Bigl( (m-w)\cdot (m-h) \Bigr) \cdot \Bigl(\frac{3}{2}w^2 h^2 \pm O(m^3)\Bigr) \right]\\
			~&=~ \frac{3}{2} \cdot \left( \sum_{w=1}^m (m-w)w^2\right) \cdot \left( \sum_{h=1}^m (m-h)h^2\right) \pm O(m^7)\\
			~&=~ \frac{3}{2} \cdot \frac{m^4}{12} \cdot \frac{m^4}{12} \pm O(m^7)\\
			~&=~ \frac{m^8}{96}  \pm O(m^7). \qedhere
	\end{align*}
\end{proof}

We can now estimate the portion of $\binom{P_m}{4}$ that gives convex sets with diagonals whose slope
have the same sign:
\[
	\frac{m^8/96 \pm O(m^7)}{\binom{(m+1)^2}{4}} ~=~ \frac{m^8/96 \pm O(m^7)}{m^8/24 \pm O(m^7)} ~=~ \frac{1}{4}\pm O(1/m)
	~\xrightarrow{m\to \infty}~ \frac{1}{4}.
\]

We can connect this last probability to the contionous model of selecting points uniformly at random 
in the unit square as follows.
Scale and translate the point set $P_m$ so that the points in $P_m$ become the center points of $(m+1)^2$ equal squares
that partition the unit square $[0,1]\times [0,1]$. Associate to every  point $q$ 
of the square $[0,1]\times [0,1]$ one of the closest points of $P_m$, denoted by $q^*$.
Now the event 
\begin{quote}
	$\mathcal{E}$: 4 uniformly and independently selected random points $q_1,q_2,q_3,q_4$ in $[0,1]\times [0,1]$
	are in convex position and both diagonals of the convex quadrilateral they define have the same color
\end{quote}
is closely related to the event  
\begin{quote}
	$\mathcal{E}'_m$: 4 uniformly and independently selected random points $q^*_1,q^*_2,q^*_3,q^*_4$ in $P_m$
	are in convex position and both diagonals of the convex quadrilateral they define have the same color.
\end{quote}
Note that in this last event the points $q^*_i$ are selected uniformly and independently from $P_m$, 
but with repetition. The difference between the probabilities of these two events, $\mathcal{E}$ and $\mathcal{E}'_m$,
tends to $0$ as $m$ increases. That is,
$\lim_{m\to \infty} |\Pr[\mathcal{E}] - \Pr[\mathcal{E}'_m]|=0$. 
Also, the difference between the probabilities of $\mathcal{E}'_m$ and
\begin{quote}
	$\mathcal{E}''_m$: 4 uniformly distinct points $\{ q^*_1,q^*_2,q^*_3,q^*_4\} \in \binom{P_m}{4}$
	are in convex position and both diagonals of the convex quadrilateral they define have the same color
\end{quote}
tends to $0$ as $m$ increases. 
Therefore 
\[
 \Pr[\mathcal{E}] ~=~ \lim_{m\to \infty} \Pr[\mathcal{E}'_m] ~=~ \lim_{m\to \infty} \Pr[\mathcal{E}''_m] ~=~ 
 \lim_{m\to \infty} \frac{1}{4}\pm O(1/m) = \frac{1}{4}.
\] 
This finishes the proof of Theorem~\ref{thm:main}.

\section{Conclusions}
We could consider a different coloring of the edges that uses a
slope criterion. 
To be precise, for any interval $I\subseteq \RR$, we can consider the edge coloring
$\chi_I$ where edges are blue if their slope is in $I$, and red otherwise. 
In this paper we have analyzed the case when $I=(0,\infty)$ and
showed that $\EE[\rec (D,\chi_{(0,+\infty)})]=\frac{1}{4} \cdot \binom{n}{4}$,
when the points are selected uniformly at random from the unit square.
We conjecture that the interval $I=(0,\infty)$ minimizes the expected number 
of monochromatic crossings, $\EE[\rec (D,\chi_I)]$, when the points are selected 
uniformly at random from the unit square. Note that one could also consider
sets $I\subset \RR$ that are not intervals.

We have made some small experimental search that justifies this conjecture.
For this, we considered several different intervals $I$ and several 
random choices of four points in the unit square, 
and then counted for each interval how many of the choices would contribute 
a monochromatic coloring. The results are consistent with the conjecture,
but not conclusive. For example, to distinguish the interval 
$(2^{-16},+\infty)$ from $(0,+\infty)$ with enough confidence, one would have
to make many repetitions and start being careful with generating the 
random points with enough precision.
The other natural candidate that uses symmetry, $I=(-1,+1)$, which
means coloring blue the near-horizontal edges, and red the near-vertical edges, 
does behave worse in the experiments.
 
It seems that one could approach the problem of computing
$\EE[\rec (D,\chi_{(\alpha,\beta)})]$ analytically, as a function 
of $\alpha$ and $\beta$, but quickly one runs into many cases that makes 
the analysis long and cumbersome.

\section*{Acknowledgments}

This work was carried out during \emph{Crossing Number Workshop 2022, Strobl, Austria.} We thank the organizers and participants for providing a fruitful research environment.

Funded in part by the Slovenian Research and Innovation Agency (P1-0297, J1-2452, N1-0218, N1-0285).
Funded in part by the European Union (ERC, KARST, project number 101071836). Views and opinions expressed are however those of the authors only and do not necessarily reflect those of the European Union or the European Research Council. Neither the European Union nor the granting authority can be held responsible for them. 
J.T. was supported by Charles University project UNCE/SCI/004. 
A.W. was supported by the Vanier Canada Graduate Scholarships program.




\bibliographystyle{abbrv}

\bibliography{col_random}



\end{document}